\documentclass[12pt]{article}
\usepackage[utf8]{inputenc}
\usepackage{amsmath}
\usepackage{amsfonts}
\usepackage{amsthm}
\usepackage{amssymb}
\usepackage{bbm}
\usepackage{color}

\title{On the Union-Closed Sets Conjecture}
\author{YINING HU \\
CNRS, Institut de Math\'ematiques de Jussieu-PRG \\
Universit\'e Pierre et Marie Curie, Case 247 \\
4 Place Jussieu \\
F-75252 Paris Cedex 05 (France) \\
{\tt yining.hu@imj-prg.fr}}
\date{}

\newtheorem{thm}{Theorem}
\newtheorem{coro}{Corollary}

\newtheorem{lem}{Lemma}
\theoremstyle{definition}
\newtheorem{defi}{Definition}

\begin{document}
\maketitle
\section{Introduction}
A collection of sets $\mathcal{A}$ is \emph{union-closed} if $S,T\in \mathcal{A}$ implies that $S\cup T \in \mathcal{A}$.
The following conjecture, often attributed to Peter Frankl, dates back to 1979. Recently Blinovsky \cite{blin} and Sch\"{a}ge \cite{Sven} 
claim to have proven the conjecture, but their proofs seem to be false. We refer the interested reader to a comprehensive overview
of this conjecture by Bruhn and Schaudt \cite{bruhn}.

 \newtheorem*{conj}{Union-closed sets Conjecture}
 \begin{conj}
 Let $\mathcal{A}$ be a union-closed finite collection of sets, containing at least one non-empty set, then there is an element which belongs to at least half of the sets in $\mathcal{A}$.
 \end{conj}

Here are some notation and convention that we will adopt in this article:
We use abbreviated notation for collections of sets of integers. For example, $\{\{1,2\},\{1,2,3\},\{3,4\}\}$ denoted by $\{12,123,34\}$. The set
$\{1,2,...,m\}$ is denoted by $[m]$, the power set of $[m]$ is denoted by $\mathcal{P}(m)$.
 For convenience we can always assume that a union-closed finite collection $\mathcal{A}$ on $n$ elements is a subset of $\mathcal{P}(n)$. 
The \emph{universe} $U(\mathcal{A})$ is the set of all elements that appear in the member sets of $\mathcal{A}$, that is, $\cup_{A\in\mathcal{A}}A$. 
 $\mathcal{A}$ is \emph{separating} if for any two distinct elements in $U(A)$, there is a set in $\mathcal{A}$ that contains one of them but does not contain the other. 
In the conjecture we could also require the collection to be separating, but this does not make a difference. 
For an element $a\in U(A)$, its frequency in $\mathcal{A}$ will be denoted by $|a|_{\mathcal{A}}$. The sub-collection of sets containing $a$ will be noted by $\mathcal{A}_a$, and its complement in $\mathcal{A}$ by $\mathcal{A}_{\bar{a}}$. Both $\mathcal{A}_a$ and $\mathcal{A}_{\bar{a}}$ are still union-closed sets. An element $b$ is said to \emph{dominate} another element $c$ in $\mathcal{A}$ if $b$ occurs in every set in $\mathcal{A}$ that contains $c$, using the notation above, this is equivalent to saying $c\notin U( \mathcal{A}_{\bar{b}})$. A set $T\in \mathcal{A}$ is said to be a \emph{basis set} if it is not the union of other sets in $\mathcal{A}$. If we remove a basis set from $\mathcal{A}$, the rest is still closed by union.

\section{On a minimal counterexample}

In all this section, let $\mathcal{A}$ be a union-closed, separating collection, where $U(\mathcal{A})=\{1,...,m\}$ 
with $|1|_\mathcal{A}\leq...\leq|m|_\mathcal{A}$. Roberts and Simpson(\cite{roberts}) have proven 
if $\mathcal{A}$ is a counterexample with the least number of sets, 
then $|\mathcal{A}|\geq 4m+1$.
Here we give an alternative proof of their result. We need to define some other notions.

For all elements $i$, we define the set 
$$A_{i}:=\cup_{\{A\in \mathcal{A}\;|\; i\notin A\}} A,$$
This set contains all elements $j$ greater than $i$, because by the assumption of separation and 
$|i|_\mathcal{A}\leq |j|_\mathcal{A}$,  there exists 
a set in $\mathcal{A}$ that contain $j$ but not $i$. 
Thus $\mathcal{A}$ contains a sub-collection $\mathcal{S}=\{ A_0=U(\mathcal{A}), A_1,..., A_{m-1} \}$ 
whose structure can be represented by the following table, where 
``1'' means that the element in the column is in the set in the row, ``0'' means that the element in the column is not in the set in the row,
``?'' means not determined. 

\vspace{5mm}

\begin{center}\begin{tabular}{c c c c c c c c c } \hline
 &$1$&$2$&$3$&\ldots&$m-3$&$m-2$&$m-1$&$m$  \\ \hline
$A_{m-1}$&?&?&?&\ldots&?&?&0&1 \\
$A_{m-2}$&?&?&?&\ldots&?&0&1&1 \\
$A_{m-3}$&?&?&?&\ldots&0&1&1&1 \\
\vdots &\vdots & & & & &\vdots &\vdots &\vdots \\
$A_{2}$&?&0&1&\ldots&1&1&1&1 \\
$A_{1}$&0&1&1&\ldots&1&1&1&1 \\
$A_{0}$ &1&1&1&\ldots&1&1&1&1 \\ \hline
\end{tabular}
\end{center}
\vspace{5mm}

If $\mathcal{A}$ is a counterexample of the union-closed sets conjecture with the least number of sets, 
then we know that $\mathcal{A}_{\bar {m}}$, which is also closed by union and has smaller cardinality than $\mathcal{A}$, 
satisfies the conjecture, thus containing an element that appears in at least half of the sets in $\mathcal{A}_{\bar {m}}$. 
It would be nice if we knew that the most frequent element in $\mathcal{A}_{\bar {m}}$ is also frequent in $\mathcal{S}$, 
because then we would know that the maximal frequency in $\mathcal{A}$ is at least the sum of its frequencies in $\mathcal{A}_{\bar {m}}$ 
and $\mathcal{S}$. Indeed we have such a result, which is the corollary of the following lemma:

\begin{lem}
\label{dom}
For $i\in\{1,...,m-1\}$, if $|i|_\mathcal{S}<m-1$, then there exists an element in $\{1,...,m-1\}$ with frequency $m-1$ in ${\mathcal{S}}$
that dominates $i$ in $\mathcal{A}$.
\end{lem}
\begin{proof}
If $|i|_\mathcal{S}<m-1$, then there exists an element $j$ in  $\{i+1,...,m-1\}$ such that $i\notin A_j$. This means that $i$ is dominated by $j$ in $\mathcal{A}$, 
as $A_j$ is by definition the union of sets in $\mathcal{A}_{\bar{j}}$. If $|j|_\mathcal{S}=m-1$, we are done. 
If not, we apply the above process to $j$ and iterate until
we find an element $k$  in $\{1,...,m-1\}$ with $|k|_{\mathcal{S}}=m-1$ that dominates $i$.

\end{proof}

\begin{coro}
\label{coro}
Among the elements of maximal frequency in a non-empty sub-collection of $\mathcal{A}$, there exists one with frequency $m-1$ in $\mathcal{S}$.
\end{coro}
\begin{proof}
Let $\mathcal{B}$ be a non-empty sub-collection of $\mathcal{A}$, and $i$ be an element of maximal frequency in $\mathcal{B}$.
If $i$ appears less than $m-1$ times in $\mathcal{S}$, then there exists an element $j$ in $\{1,...,m-1\}$ with
frequency $|j|_\mathcal{S}=m-1$  that dominates $i$ in $\mathcal{A}$, thus in $\mathcal{B}$.
\end{proof}

 If $\mathcal{A}$ is a minimal counterexample of the conjecture, then $|\mathcal{A}|=2n+1$ for some integer $n$. In fact, if $\mathcal{B}$ is a counterexample with $|\mathcal{B}|=2n+2$, then we know that the maximal frequency in $\mathcal{B}$ is less than $n+1$. If we remove a basis set from $\mathcal{B}$, then we get a union-closed collection with $2n+1$ sets whose maximal frequency is still less than $n+1$, which makes a yet smaller counterexample. If $a$ is an element of maximal frequency in $\mathcal{A}$, then $|a|_{\mathcal{A}}=n$. This is because if the frequency of $a$ were less than $n$, then we could remove a basis set from $\mathcal{A}$ and get a smaller counterexample. \\
 \\
 Now we know that if $\mathcal{A}$ is a minimal counterexample, then in $\mathcal{A}_{\bar {m}}$, there are elements that occur in at least half of the sets, and among these elements, there is at least one with frequency $m-1$ in $\mathcal{S}$. Thus $\mathcal{A}_m$ must contain more than $2m-2$ elements for $\mathcal{A}$ to be a counterexample. This, combined with the discussion above, leads to the following theorem:

\begin{thm}\label{min}
If a separating, union-closed collection $\mathcal{A}$ is a counterexample of the union-closed sets conjecture of minimal cardinality, then $|\mathcal{A}|\geq 4m-1$, where $m=|U( \mathcal{A})|$.
\end{thm}

 \begin{proof}
As $\mathcal{A}$ is a minimal counterexample, $|\mathcal{A}|=2n+1$ for some integer n. $\mathcal{A}_{\bar{m}}$ is union-closed and $|\mathcal{A}_{\bar{m}}|=n+1$, therefore the maximal frequency in $\mathcal{A}_{\bar{m}}$ is at least $\frac{n+1}{2}$ by the minimality of $\mathcal{A}$. By Corollary \ref{coro}, there exists $a\in U( \mathcal{A})$ such that $|a|_{\mathcal{A}_{\bar{m}}}\geq \frac{n+1}{2}$ and $|a|_\mathcal{S}=m-1$. Also we must have $|a|_{\mathcal{A}_{\bar{m}}}+|a|_S\leq |a|_\mathcal{A}\leq n$, i.e., $\frac{n+1}{2}+m-1\leq n$. Therefore $n\geq 2m-1$ and $2n+1 \geq 4m-1$.
\end{proof}

Bo\v{s}njak and Markovi\'{c} \cite{bos} have proved that a minimal counterexample has $m\geq 12$ this number is improved by
\v{Z}ivkovi\'c and Vu\v{c}kovi\'c \cite{12} to 13. Thus Theorem \ref{min} implies that a minimal counterexample contains at least 51 sets. 
We always assume separation when we talk about the relation of the size of the counterexample and the size of the universe,
so that we could not just duplicate elements and make the size of the universe arbitrarily big with essentially the same collection.\\

\section{The $\varepsilon$-union-closed sets conjecture}

As the union-closed sets conjecture has proven to be a difficult problem, we may want to try to prove the following weakened conjecture:

\newtheorem{weak}{$\varepsilon$-Union-Closed Sets Conjecture}
\begin{weak}
  There exists $\varepsilon>0$ such that for all union-closed finite collection of sets $\mathcal{A}$ containing at least one non-empty set, 
  there is an element which belongs to at least $\varepsilon\cdot |\mathcal{A}|$ of the sets in  $\mathcal{A}$.
\end{weak}


In the last section we have proven that if $\mathcal{A}$ is a minimal counterexample of the union-closed sets conjecture,
then $|\mathcal{A}|\geq 4\cdot|U(\mathcal{A})|-1$. For a counterexample that is not necessarily minimal,  
what do we know about the relation of the size of the collection and the size of the universe? 
From the construction of the sub-collection $\mathcal{S}$ in the last section, we know that
for any union-closed finite collection of sets $\mathcal{B}$, there is an element of frequency at least $|U(\mathcal{B})|$. Thus the union-closed
sets conjecture holds for all union-closed finite collections $\mathcal{B}$ with $|\mathcal{B}|\leq 2\cdot |U(\mathcal{B})|$. 
This condition is fairly easy to establish. But can we obtain a better bound, for example with $2+\varepsilon'$ instead of $2$?
There is reason to believe that this would not be easy, as it would imply the $\varepsilon$-union-closed sets conjecture:

\begin{thm}
Let $c>2$. If the union-closed sets conjecture is true for all separating, union-closed finite collection $\mathcal{A}$ with 
$|\mathcal{A}|\leq c \cdot|U( \mathcal{A})|$, then for all union closed families $\mathcal{B}$, there exists $x\in U(\mathcal{B})$ 
with $|x|_{\mathcal{B}}\geq \frac{c-2}{2(c-1)}\cdot|U( \mathcal{B})|$.
\end{thm}
\begin{proof}
Suppose that for $c>2$, the union-closed sets conjecture is true for all separating, union-closed finite collection $\mathcal{A}$ 
such that $|\mathcal{A}|\leq c \cdot|U( \mathcal{A}|)$. 

Let $\mathcal{B}$ be a union-closed finite collection with $|U(\mathcal{B})|=m$ and $|\mathcal{B}|=n$. 
If $n\leq c\cdot m$ then the conclusion is true. Suppose now that $n> c\cdot m$. Let $p$ be a positive integer whose value will be made precise later.
We construct another collection $\mathcal{C}$ by adding
$p$ new elements to $U(\mathcal{B})$ and $p$ new sets to $\mathcal{B}$:
$$U(\mathcal{C}):=U(\mathcal{B})\cup \{x_1,...,x_p\},$$
$$\mathcal{C}:= \mathcal{B}\cup \{U(\mathcal{C})\backslash \{x_i\}\;|\;i=1,...,p-1\} \cup \{U(\mathcal{C})\}.$$
$\mathcal{C}$ is still a union-closed, separating collection.
In order to apply the assumption to $\mathcal{C}$, we need $p$ to satisfy 
$$\frac{|\mathcal{C}|}{|U(\mathcal{C})|}=\frac{n+p}{m+p}\leq c,$$
that is, 
$$p\geq \frac{n-cm}{c-1}.$$
On the other hand, as we will see shortly after, we want $p$ to be as small as possible. So we choose
$$p= \left\lceil\frac{n-cm}{c-1}\right\rceil<\frac{n-cm}{c-1}+1.$$
Now by the assumption and the choice of $p$, we know that there is an element in $U(\mathcal{C})$ that appears in at least 
$\left\lceil \frac{n+p}{2}\right\rceil$ sets in $\mathcal{C}$. As $n>p$, this element cannot be one of the $p$ added elements, so it is
in $U(\mathcal{B})$. Its frequency in $\mathcal{B}$ is $\left\lceil\frac{n-p}{2}\right\rceil$. We have
\begin{align*}
 \frac{n-p}{2n}&>\frac{n-\frac{n-cm}{c-1}-1}{2n}\\
 &=\frac{(c-1)n-n+cm-c+1}{2n(c-1)}\\
 &>\frac{1}{2}-\frac{1}{2(c-1)}\\
 &=\frac{c-2}{2(c-1)},
\end{align*}
which ends the proof.


\end{proof}

\section{A bound for the minimal maximal frequency}

We define a function $\phi : \mathbbm{N}^*\rightarrow \mathbbm{N}^*$, where $\phi (n)$ is the minimum of maximal frequencies of union-closed collections over $n$ sets:
$$\phi(n)=\min_{\scriptscriptstyle{\mathcal{A} \mbox{\tiny  union-closed, } |\mathcal{A}|=n}}\max_{\scriptscriptstyle{a\in U(\mathcal{A})}}|a|_{\mathcal{A}}$$
The union-closed sets conjecture can be expressed as $\phi(n)\geq \frac{1}{2} \cdot n$.

\subsection{Renaud's construction and boundary function}
Renaud and Fitina \cite{renaud} has conjectured that $\phi(n)$ is equal to Conway's challenge sequence $(a(n))$ defined as:
$$a(1)=a(2)=1, \;\;a(n)=a(a(n-1))+a(n-a(n-1))$$
Mallows \cite{mallows} has proved that the sequence $(a(n))$ has the property that $a(n)\geq  \frac{1}{2} \cdot n$. 
We also know that $a(n+1)\in \{a(n), a(n)+1\}$. Renaud and Fitina \cite{renaud} have proved that $\phi(n)$ has the same property and that
$\phi(n)\leq a(n)$ by constructing a union-closed collection with maximal frequency $a(n)$ for all $n\geq 2$.

Renaud \cite{ren} has calculated the values of $\phi(n)$ for $n=1,...,18$. The values coincide with that of $a(n)$.
But at $n=23$, he has found an counterexample \cite{ren2}. Using abbreviated notation, the union-closed collection 
$$\mathcal{P}(4)\cup \{12345,1235,1245,1345,2345,125,345 \}$$ has highest element frequency 13, whereas $a(23)=14$.

Renaud \cite{ren2} then defined another function $\beta: \mathbb{N}^*\rightarrow \mathbb{N}^*$ with the 
property $\phi(n)\leq\beta(n)\leq a(n)$, whose value corresponds to the maximal frequency of the 
union-closed collection $\mathcal{B}(n)$ of $n$ sets, constructed in the following way:
let $k$ be an integer such that  $2^{k-1}<n\leq 2^k$. We obtain 
$\mathcal{B}(n)$ by deleting sets containing $k$ from $\mathcal{P}(k)$, following rules: a smaller set is always deleted before a larger set;
sets of the same size are deleted in an order such that the frequency of the element $1,...,k-1$ in the remaining sets are ``balanced'', that is, 
the difference of their frequencies is at most 1. Therefore if
$$n=2^k-\sum\limits_{i=1}^{r-1}\binom{k-1}{i}-v$$
where $0\leq r \leq k-1$ and $0\leq v<\binom{k-1}{r}$, then 
$$\beta(n)=2^{k-1}-\sum\limits_{i=1}^{r-1}\binom{k-2}{i-1}-\left\lfloor\frac{r\cdot v}{k-1}\right\rfloor.$$

For example $\mathcal{B}(23)=\mathcal{P}(4)\cup \{12345,1235,1245,1345,2345,125,345 \}$ is the collection mentioned above, and
$\beta (23)=13<14=\alpha(23)$. Therefore $\beta(n)$ is a better boundary function of $\phi(n)$. But $\beta(n)$ is not optimal either. Renaud gives
the example of the familly 
$$\mathcal{P}(6)\backslash\{6,5,16,15,36,45,136,145\}$$
in which the most frequent element appears in 30 sets, but $\beta(56)=31$.

\subsection{An improved bound of $\phi(n)$}
Here we give another way of constructing union-closed collections whose maximal frequency approximates $\phi(n)$ better, and we show that
the gap between $\phi(n)$ and $\beta(n)$ is not bounded. We make use of the following notion in our construction:

\begin{defi}
An up-set $\mathcal{U}$ on m elements is a subset of $\mathcal{P}(m)$ such that $S\in \mathcal{U}$ and $S\subset T \in [m]$ 
implies that $T\in \mathcal{U}$.
\end{defi}

We note that in the construction of the up-set of Renaud, the smaller sets are discarded while the larger sets are kept. 
We could have a lower maximal frequency if we could keep more small sets while keeping frequency ``balanced'' among the elements. 
For example, consider the the union-closed collection $\mathcal{C}$ composed of $\mathcal{P}(12)$ and the up-set on $13$ elements
generated by the sets $\{1,2,3,4,13\}$, $\{5,6,7,8,13\}$, $\{9,10,11,12,13\}$.
In the up-set, there are ``holes'' at level 6 to 11 (which does not happen in the construction of Renaud), 
that is,  we will not generate all sets of size $6$ to 11 in $\mathcal{P}(13)\backslash \mathcal{P}(12)$.
At the same time, all the elements in $\{1,2,...,12\}$ have the same frequency by symmetry. Therefore the maximal frequency in $\mathcal{C}$ is
$$2^{12}+\frac{\sum_{C\in\mathcal{C}\backslash {\mathcal{P}(12)}} (|C|-1)}{12} $$
Let $\mathcal{B}=\mathcal{B}(|\mathcal{C}|)$ be the union-closed collection with the same number of sets as $\mathcal{C}$ using the construction of
Renaud,  the maximal frequency in $\mathcal{B}$ is 
$$2^{12}+\left\lceil \frac{\sum_{B\in\mathcal{B}\backslash {\mathcal{P}(12)}} (|B|-1)}{12} \right\rceil$$
As $\mathcal{B}$ and $\mathcal{C}$ have the same number of sets but $\mathcal{B}$ contains more larger sets, 
the maximal frequency of $\mathcal{C}$ is smaller than that of $\mathcal{B}$.

More generally, let $s$ and $k$ be integers greater than 1, and not both equal to 2.
Let $\mathcal{C}_{s,k}$ be the union of $\mathcal{P}(sk)$ and the up-set on $sk+1$ elements generated by $\{1,2,\ldots,s,sk+1\},\;\{s+1,s+2,\ldots, 2s,sk+1\},\ldots,\; \{(k-1)s+1, (k-1)s+2, \ldots, ks,sk+1\}$.
In the up-set, there are holes at level $s+2$ to $k(s-1)+1$. Therefore there exists a bijective function $f$ from $\mathcal{C}_{s,k}$ to $\mathcal{B}(|\mathcal{C}_{s,k}|)$ such that for all $C\in \mathcal{C}_{s,k} $,
$|C|\leq |f(C)|$ and $\exists C\in \mathcal{C}$ such that $|C|<|f(C)|$. As in $\mathcal{C}$  the elements $1,2,...,sk$ have the same frequency by symmetry
 and the frequency of the element $sk+1$ is the same in $\mathcal{C}$ and $\mathcal{B}(\mathcal{C})$, the maximal frequency in $\mathcal{C}$ is smaller than that in $\mathcal{B}(|\mathcal{C}|)$.



To show that the gap between the maximal frequency of this construction and that of Renaud is not bounded,  
consider $\mathcal{C}_{2,N}$, which is the union of $\mathcal{P}(2N)$ and the up-set 
generated by the sets $\{1,...,N,2N+1\}$ and $\{N+1,...,2N,2N+1\}$. 
This up-set contains $2^{N+1}-1$ sets, and the frequency of an element in $\{1,2,...,2N\}$ is $2^N+2^{N-1}-1$.
The construction of Renaud $\mathcal{B}(2^{2N}+2^{N+1}-1)=\mathcal{P}(2N)\cup \mathcal{U}$, where $\mathcal{U}$ is an up-set on $2N+1$ elements. 
According to Mitzenmacher and Upfal \cite{coeff}, for all $k\in \mathbbm{N}$ and $k<2N$, 
$\binom{2N}{k}\geq \frac{1}{2N+1}\cdot 2^{H(k/2N)\cdot 2N}$, where $H$ is the binary entropy defined by 
$H(p)=-p\cdot \log_2(p)-(1-p)\log_2(1-p)$. 
For $N$ big enough and  $k=\left\lceil \frac{2N}{5}\right\rceil$, $H(k/2N)>H(0.2)>0.7$. Therefore
$$\binom{2N}{k}>\frac{1}{2N+1}\cdot 2^{H(k/2N)\cdot 2N}>\frac{1}{2N+1}\cdot 2^{1.4N}>2^{N+1}.$$
This means that in $\mathcal{U}$, there are no sets of size less than $2N-k=\left\lfloor\frac{8N}{5}\right\rfloor $. 
Therefore the maximal frequency among the elements $\{1,2,...,2N\}$ in $\mathcal{U}$ is at least 
$$\frac{\left\lfloor\frac{8N}{5}\right\rfloor-1}{2N}\cdot (2^{N+1}-1).$$
The gap between this number and $2^N+2^{N-1}-1$ goes up to infinity as $N$ approaches infinity.

The condition that $s$ and $k$ are both greater than one and not both equal to two is not always satisfied, for example when $n=56$.
In this case we obtain the same union-closed collection as Renaud. We have seen in the last subsection that $\beta(56)>\phi(56)$.
This means that our construction is not optimal either. Intuitively, this can be explained by the fact that in both constructions with $n+1$ 
elements,
while the frequency of 
elements $\{1,...,n\}$ is ``balanced'', the frequency of the element $n+1$ is too low, which leaves some space for further ``compression''.



\begin{thebibliography}{99}


\bibitem{blin}V. Blinovsky, ``Proof of Union- Closed Sets Conjecture'',  arXiv:1507.01270 .
\bibitem{bos}
I.Bo\v{s}njak and P. Markovi\'{c}, The 11-element case of Frankl's conjecture, \emph{Electronic J. Combinatorics 15 (2008)}, \#R88
\bibitem{bruhn}
H. Bruhn and O. Schaudt, ``The journey of the union-closed sets conjecture'', arXiv:1309.3297.
\bibitem{mallows}
C.I. Mallows, Conway's challenge sequence, \emph{American Mathematical Monthly 98} (1991), 5--20.
\bibitem{coeff}
M. Mitzenmacher and E. Upfal, Probability and computing: Randomized algorithms and probabilistic analysis, \emph{Cambridge University Press}, 2005.
\bibitem{ren2}
J.C. Renaud, A second approximation to the boundary function on union-closed collections, \emph{Ars Combin.} 41 (1995) 177--188.

\bibitem{ren}
J.C. Renaud, Is the Union-closed sets conjecture the best possible? \emph{Journal of the Australian Mathematical Society (Series A)} 51 (1991), 276--283.
\bibitem{renaud}
J-C. Renaud and L.F. Fitina, On union-closed sets and Conway's sequence, \emph{Bulletin of the Australian Mathematical Society} 47 No.2 (1993), 321--332.


\bibitem{roberts}
I. Roberts and J. Simpson, A note on the union-closed sets conjecture, \emph{Australas. J. Comb. 47}, 265--267
(2010)
\bibitem{Sven}S. Sch\"{a}ge, ``On the Union-Closed Set Conjecture'',  arXiv:1607.01007.


\bibitem{12}
 M. \v{Z}ivkovi\'c and B. Vu\v{c}kovi\'c, ``The 12 element case of Frankl's conjecture'', preprint, 2012.
\end{thebibliography}
\end{document}